\documentclass[12pt]{article}
\usepackage{hyperref}
\usepackage{amsmath}
\usepackage{amsfonts}
\usepackage{amsthm}
\usepackage{indentfirst}
\usepackage{geometry}
\usepackage{color}

\renewcommand{\theequation}{\thesection.\arabic{equation}}
\usepackage{indentfirst}
\date{}
\makeatletter
\@addtoreset{equation}{section}

\newtheorem{theorem}{Theorem}[section]
\newtheorem{lemma}[theorem]{Lemma}

\begin{document}
    \begin{center}
    {\Large  Radial symmetry of positive solutions of an integral system
    associated with the reversed Stein-Weiss inequality}
\end{center}

\vskip 5mm

\begin{center}
    {\sc Tiantian Zhou \quad and \quad Yutian Lei }
\end{center}

\vskip 5mm {\leftskip5mm\rightskip5mm \normalsize
    \noindent{\bf{Abstract}}
    Whether the solutions of conformal equations in the whole space
    are radially symmetric is an interesting topic. Chen-Li-Ou
    proved the radial symmetry for integral systems of the Hardy-Littlewood-Sobolev
    type and the Stein-Weiss type by the method of moving planes in integral form.
    In 2015, Dou-Zhu obtained the radial symmetry of extremal functions of the
    reversed Hardy-Littlewood-Sobolev inequality by the method of moving spheres,
    and Liu proved the radial symmetry of solutions of the Euler-Lagrange system
    by the method of moving planes developed by Dou-Guo-Zhu.
    In this paper, we also use the method of moving planes to prove the radial symmetry
    of positive solutions of the Euler-Lagrange system satisfied by the extremal functions
    of the reversed Stein-Weiss inequality established by Chen-Liu-Lu-Tao in 2018.

    }

\renewcommand{\theequation}{\thesection.\arabic{equation}}
\catcode`@=11
\@addtoreset{equation}{section}
\catcode`@=12

    \section{Introduction}

    The Hardy-Littlewood-Sobolev inequality states that (cf. Theorem 1
    in Chapter 5 of \cite{S} or Theorem 4.3 in \cite{LL})
    \begin{equation}\label{hls}
    \left|\int_{\mathbb{R}^n}\int_{\mathbb{R}^n}\frac{f(x)g(y)}{|x-y|^{\lambda}}dxdy \right|
    \leq C_{n,\lambda}\|f\|_{L^{r}(\mathbb{R}^n)}\|g\|_{L^s(\mathbb{R}^n)}
    \end{equation}
    holds for all $(f,g)\in L^{r}(\mathbb{R}^n)\times L^s(\mathbb{R}^n)$,
    where $0<\lambda<n$, $1<r,s<\infty$, and $1/r+1/s+\lambda/n=2$.
    The best constants and extremal functions
    play important roles in the theory of nonlinear PDEs and mathematical physics.
    For example, $C_{n,\lambda}$ can be used to estimate the upper bound of Coulomb energy
    in the Thomas-Fermi model (cf. \cite{Lieb}).
    It is also associated with the study of steady-state solutions of porous media equations (cf. \cite{CCL}).
    In 1983, Lieb \cite{LEH} proved the existence of extremal functions by using rearrangement argument,
    and then obtained the best constant of \eqref{hls} in the case of $r=s={2n}/(2n-\lambda)$.
    Now, the extremal function is
    \begin{equation}\label{Seq}
    u(x)=\left(\frac{a}{d+|x-\tilde{x}|^2}\right)^{\lambda/2},
    \end{equation}
    where $a, d>0$ and $\tilde{x}\in \mathbb{R}^n$.
    Clearly, the Euler-Lagrange equation satisfied by the extremal functions is the conformal equation
    \begin{equation}\label{eq}
        u(x)=\int_{\mathbb{R}^n}\frac{u^p(y)}{|x-y|^{\lambda}}dy, \quad u>0\quad in \quad \mathbb{R}^n.
    \end{equation}
    Here $p=(2n-\lambda)/\lambda$.
    A question posed by Lieb is whether the solutions of \eqref{eq} are unique.
    Chen, Li and Ou \cite{CLO} and Li \cite{LYY} gave independently the positive answer
    and classified the $L_{loc}^{2n/\lambda}(\mathbb{R}^n)$-solutions of \eqref{eq} as the form of \eqref{Seq}.
    In their work, the method of moving planes and the method of moving spheres come into plays.

    When $r \neq s$, the Euler-Lagrange system is
    \begin{equation}\label{el}
        \begin{cases}
            u(x)=\displaystyle\int_{\mathbb{R}^n}\frac{v^q(y)}{|x-y|^{\lambda}}dy,
            &\text{ $ u>0 \quad in \quad \mathbb{R}^n, $ } \\
            v(x)=\displaystyle\int_{\mathbb{R}^n}\frac{u^p(y)}{|x-y|^{\lambda}}dy,
            &\text{ $ v>0 \quad in \quad \mathbb{R}^n $ }
        \end{cases}
    \end{equation}
    Here $p=1/(r-1)$ and $q=1/(s-1)$ satisfy the critical condition of Sobolev type
    \begin{equation}\label{cc}
        \frac{1}{p+1}+\frac{1}{q+1}=\frac{\lambda}{n}.
    \end{equation}
    Chen, Li and Ou \cite{CLO5} applied the method of moving planes in integral form to prove that
    the finite energy solutions are radially symmetric and decreasing about some point in $\mathbb{R}^n$.

    In 1958, Stein and Weiss proved the weighted Hardy-Littlewood-Sobolev inequality
    \begin{equation}\label{888}
    \left|\int_{\mathbb{R}^n}\int_{\mathbb{R}^n}
    \frac{f(x)g(y)}{|x|^{\alpha}|x-y|^{\lambda}|y|^{\beta}}dxdy\right|
    \leq C_{n,\alpha,\beta,p,q'}\|f\|_{L^{q'}(\mathbb{R}^n)}\|g\|_{L^p(\mathbb{R}^n)},
    \end{equation}
    where $(f,g)\in L^{r}(\mathbb{R}^n)\times L^s(\mathbb{R}^n)$,
    and $p$, $q'$, $\alpha$, $\beta$ and $\lambda$ satisfy the following conditions:
    $$\begin{cases}
    0<\lambda<n, \ \alpha+\beta \geq 0, \ \alpha<n/q, \ \beta<n/p',\\
    1/q'+1/p \leq 1, \ \ 1/q'+1/p+(\alpha+\beta+\lambda)/n=2.
    \end{cases}
    $$
    Here $p'$ and $q$ are the H\"older conjugates of $p$ and $q'$ respectively.
    In the case of $q'=p$ or $q'+p=2$, Lieb \cite{LEH} obtained the existence of extremal functions.
    Beckner \cite{BW1} gave a determined upper bound of the best constant when $p=q$, $1<p, q'<\infty$, $\alpha=0$,
    $\lambda=n-\alpha/q'$ and $\beta=\alpha/q'$.

    The Euler-Lagrange system satisfied by extremal functions of \eqref{888} is
    \begin{equation}\label{jj3}
        \left \{
        \begin{array}{lll}
            u(x)=\displaystyle\int_{\mathbb{R}^n}\frac{v^{p_2}(y)dy}
            {|x|^{\alpha}|x-y|^{\lambda }|y|^{\beta}},   \\
            v(x)=\displaystyle\int_{\mathbb{R}^n}\frac{u^{p_1}(y)dy}
            {|x|^{\beta}|x-y|^{\lambda }|y|^{\alpha}}.
        \end{array}
        \right.
    \end{equation}
    In 2006, Jin and Li \cite{JL} used the method of moving planes in integral form
    to prove that the finite energy solutions of \eqref{jj3} are radially symmetric
    and decreasing about some point. Moreover, Chen and Li
    obtained in \cite{CL} the best constant of \eqref{888} in the case of $p_1=p_2$
    and $\alpha=\beta$. In 2019, Chen et al. \cite{CLT} found extremal functions of \eqref{888}
    with full weights and with horizontal weights in the Heisenberg group.
    In addition, \cite{JL2} and \cite{LLM} studied the integrability and asymptotic behavior
    of the finite energy solutions of \eqref{jj3}.

    In 2015, Dou and Zhu proved in \cite{DZ} the following reversed Hardy-Littlewood-Sobolev
    inequality (see also \cite{BW3,NN})
    $$
    \left|\int_{\mathbb{R}^n}\int_{\mathbb{R}^n}\frac{f(x)g(y)}{|x-y|^{\lambda}}dxdx\right|
    \geq C\|f\|_{L^r(\mathbb{R}^n)}\|g\|_{L^s(\mathbb{R}^n)},
    \quad \forall (f,g)\in L^r(\mathbb{R}^n)\times L^s(\mathbb{R}^n),
    $$
    where
    $$
    n\geq 1, \ \lambda<0, \ 0<r,s<1, \ 1/r+1/s+\lambda/n=2.
    $$
    In addition, they proved the existence of extremal functions and classified
    the extremal functions by the method of moving spheres when $f \equiv g$ and $r=s$.
    The Euler-Lagrange system satisfied by those extremal functions is
    \begin{equation}\label{j2}
        \left \{
        \begin{array}{lll}
            u(x)=\displaystyle\int_{\mathbb{R}^n}|x-y|^{\Lambda }v^{-p_2}(y)dy,  \\
            v(x)=\displaystyle\int_{\mathbb{R}^n}|x-y|^{\Lambda }u^{-p_1}(y)dy.
        \end{array}
        \right.
    \end{equation}
    Here $\Lambda=-\lambda>0$, $p_1=1/(1-r)$ and $p_2=1/(1-s)$.
    Paper \cite{L} shows the asymptotic estimates and the existence results of solutions of \eqref{j2}.
    But the author did not obtain the radial symmetry of those solutions.
    Thanks to the work of Dou, Guo and Zhu (cf. \cite{DGZ}).
    They developed the method of moving planes in integral form,
    and then Liu \cite{LZ} used their scheme
    to obtain the radial symmetry and monotonicity of solutions of \eqref{j2}.

    When $u\equiv v$, and $p_1\equiv p_2$, \eqref{j2} is reduced to
    \begin{equation}\label{j7}
        u(x)=\int_{\mathbb{R}^n}|x-y|^{\Lambda }u^{-p_1}(y)dy.
    \end{equation}
    In 2004, Li \cite{LYY} used the method of moving spheres to classify the solutions as the form of
    $$
    u(x)=\left(\frac{d+|x-\tilde{x}|^2}{a}\right)^{\Lambda/2}
    $$
    when $p_1=1+2n/\Lambda$. In addition, he posed an question: whether has
    \eqref{j7} no positive (regular) solutions for all $n\geq 1$, $\Lambda>0$ and $p_1>1+2n/\Lambda$?
    In 2007, Xu confirmed this conclusion by the Pohozaev identity in integral form (cf. \cite{XXW}).
    Other related results on PDE with negative exponents can be found in \cite{HW,MW} and references therein.

    In 2018, Chen et al. \cite{CLLT} proved that the following reversed Stein-Weiss inequality
    \begin{equation}\label{j6}
        \left|\int_{\mathbb{R}^n}\int_{\mathbb{R}^n}|x|^{\alpha}|x-y|^{\Lambda}f(x)g(y)|y|^{\beta}dxdy\right|
        \geq C_{n,\alpha,\beta,p,q'}\|f\|_{L^{q'}(\mathbb{R}^n)}\|g\|_{L^p(\mathbb{R}^n)}
    \end{equation}
    holds for any non-negative functions $f \in L^{q'}(\mathbb{R}^n)$ and $g\in L^p(\mathbb{R}^n)$,
    where
    $$
    n\geq 1, \ p \in (0,1), \ q'\in (0,1), \ \Lambda>0, \ 0\leq \alpha<-{n}/{q}, \ 0\leq \beta<-{n}/{p'},
    $$
    and
        \begin{equation}\label{j5}
            \frac{1}{p}+\frac{1}{q'}-\frac{\alpha+\beta+\Lambda}{n}=2.
        \end{equation}
    They also verified that the best constant can be achieved by minimizing the functional
    $$
    J(f,g)=\int_{\mathbb{R}^n}\int_{\mathbb{R}^n}|x|^{\alpha}|x-y|^{\lambda}f(x)g(y)|y|^{\beta}dxdy
    $$
    under the condition $\|f\|_{L^{q'}(\mathbb{R}^n)}=\|g\|_{L^q(\mathbb{R}^n)}=1$.
    Set $u=c_1f^{q'-1}$, $v=c_2g^{p-1}$, $1/(q'-1)=-p_1$ and $1/(p-1)=-p_2$.
    By choosing proper $c_1$ and $c_2$, one can obtain the Euler-Lagrange system
    \begin{equation}\label{j3}
        \left \{
        \begin{array}{lll}
            u(x)=\displaystyle\int_{\mathbb{R}^n}|x|^{\alpha}|x-y|^{\Lambda }v^{-p_2}(y)|y|^{\beta}dy,   \\
            v(x)=\displaystyle\int_{\mathbb{R}^n}|x|^{\beta}|x-y|^{\Lambda }u^{-p_1}(y)|y|^{\alpha}dy.
        \end{array}
        \right.
    \end{equation}
    Here \eqref{j5} becomes
    $$
    \frac{1}{p_1-1}+\frac{1}{p_2-1}=\frac{\alpha +\beta +\Lambda}{n}.
    $$

    Motivated by the work above, we will prove the radial symmetry of positive solutions of \eqref{j3} in this paper.
    We call $(u,v)$ a pair of positive solution of \eqref{j3},
    if $(u,v)$ solves \eqref{j3} a.e. in $\mathbb{R}^n$ and belongs to
    $$
    \{(u, v) \in L_{\rm loc}^\infty(\mathbb{R}^n) \times L_{\rm loc}^\infty(\mathbb{R}^n);
    u\geq 0,\; v\geq 0 \;\; in \; \mathbb{R}^n, \; and \;  u>0, \ v>0 \; when \; x \neq 0\}.
    $$

    The main result in this paper is the following theorem.

    \begin{theorem}\label{Rth3}
        Let $(u,v)$ be a pair of positive solutions of \eqref{j3},
        where
        \begin{equation}\label{j44}
        n \geq 1, \ \Lambda>0, \ \min\{p_1, p_2\}>1, \ 0\leq \alpha<{n}/(p_1-1),
        \ 0\leq \beta<{n}/(p_2-1),
        \end{equation}
        and
        \begin{equation}\label{j4}
            \frac{1}{p_1-1}+\frac{1}{p_2-1}\leq \frac{\alpha +\beta +\Lambda}{n}.
        \end{equation}
        Then $u$ and $v$ are radially symmetric and increasing about the origin.
    \end{theorem}

  Here we use the method of moving planes in integral form introduced by Chen-Li-Ou in \cite{CLO},
  which was also applied in \cite{JL} to prove
  the radial symmetry of positive solutions of the Stein-Weiss-type system \eqref{jj3}.
  Although the Stein-Weiss inequality does not work for \eqref{j3} as in \cite{JL}, we can
  use the scheme posed by Dou-Guo-Zhu in \cite{DGZ} which had been employed successfully
  to deal with the radial symmetry of solutions of \eqref{j2} in \cite{LZ}.
  However, different from \eqref{j2}, \eqref{j3} is a double weighted system.
  It seems difficult to directly obtain the radial symmetry and monotonicity of $u$ and $v$.
  We introduce the auxiliary functions
  \begin{equation}\label{vb6}
  t(x):=|x|^{-\alpha}u(x), \; w(x):=|x|^{-\alpha/p_1}u(x), \;
  s(x):=|x|^{-\beta}v(x),
  \end{equation}
    and demonstrate the radial symmetry and monotonicity of $w$ and $s$ by the argument of moving planes.
    This implies that $u$ and $v$ are also radially symmetric and monotonous about the origin.
Therefore, we have to handle the singularity near the origin when estimating the integrals involving $w$ and $s$.
 Here we use the ideas in \cite{CLO} where the singularity
  near the origin comes from the Kelvin transformation though \eqref{eq} is without weights.

    \section{Preliminaries}

    The following lemma, which was proved in \cite{CLLT},
    indicates asymptotic estimation of the solutions to \eqref{j3}.

    \begin{lemma}\label{le1} (Lemma 14 and Theorem 3 in \cite{CLLT})
        Let $\alpha, \beta ,p_1 ,p_2, \Lambda$ be positive.
        Assume that $(u, v)$ is a pair of positive solutions of \eqref{j3}.
        Then

        (i) we have
        \begin{equation}\label{600}
            \int_{\mathbb{R}^n}(1+|y|^{\Lambda })v^{-p_2}(y)|y|^{\beta}dy<\infty,\quad
            \int_{\mathbb{R}^n}(1+|y|^{\Lambda })u^{-p_1}(y)|y|^{\alpha}dy<\infty.
        \end{equation}

        (ii) There exist some constant $C_1$, $C_2>0$ such that for all $x \in \mathbb{R}^n$,
        \begin{equation}\label{601}
        \left \{
        \begin{array}{lll}
            &{C_1}^{-1}(1+|x|^{\Lambda})\leq \frac{u(x)}{|x|^{\alpha}}\leq C_1(1+|x|^{\Lambda})\\[3mm]
            &{C_2}^{-1}(1+|x|^{\Lambda})\leq \frac{v(x)}{|x|^{\beta}}\leq C_2(1+|x|^{\Lambda}).
        \end{array}
        \right.
        \end{equation}

        (iii) There hold
        \begin{equation}\label{602}
            \begin{cases}
              \displaystyle  0<a:=\lim_{|x|\to \infty}\frac{u(x)}{|x|^{\Lambda+\alpha }}
                =\int_{\mathbb{R}^n}v^{-p_2}(y)|y|^{\beta}dy<\infty,\\
              \displaystyle  0<b:=\lim_{|x|\to \infty}\frac{v(x)}{|x|^{\Lambda+\beta }}
                =\int_{\mathbb{R}^n}u^{-p_1}(y)|y|^{\alpha}dy<\infty.
            \end{cases}
        \end{equation}
    \end{lemma}

    \begin{lemma} \label{lem3}
        Assume that $(u, v)$ is a pair of positive solutions of \eqref{j3} with \eqref{j44} and \eqref{j4}.
        Then both $t(x)$ and $s(x)$ in \eqref{vb6} are continuous in $\mathbb{R}^n$,
        and differentiable in $\mathbb{R}^n \setminus \{0\}$.
        In addition, there hold
        \begin{equation}\label{td}
        \nabla t(x)= \int_{\mathbb{R}^n}(\nabla_x|x-y|^{\Lambda})v^{-p_2}(y)|y|^{\beta}dy,
        \quad as \ \ x \neq 0;
        \end{equation}
        \begin{equation}\label{sd}
        \nabla s(x)= \int_{\mathbb{R}^n}(\nabla_x|x-y|^{\Lambda})u^{-p_1}(y)|y|^{\alpha}dy,
        \quad as \ \ x \neq 0.
        \end{equation}
    \end{lemma}

    \begin{proof}
        We here only prove the conclusions about $t$. The proof of the conclusions about $s$ is analogous.

        First we claim that $t$ is continuous at the origin. In fact, when $x \in B_1(0)$,
        by \eqref{600} there holds
        $$
        |x-y|^\Lambda v^{-p_2}(y)|y|^\beta
        \leq (1+|y|^\Lambda) v^{-p_2}(y)|y|^\beta \in L^1(\mathbb{R}^n).
        $$
        Thus, by the dominated convergence theorem, we have the following continuity result
        $$
        \lim_{x \to 0}t(x)=\int_{\mathbb{R}^n} v^{-p_2}(y)|y|^{\Lambda+\beta}dy=t(0).
        $$

        Next, we prove $t$ is differentiable at $x \in \mathbb{R}^n \setminus \{0\}$.

    \textit{Case 1.} When $\Lambda \geq 1$, we can see that the defects of the improper integral
    $$
    F(x):=\int_{\mathbb{R}^n}[\nabla_x |x-y|^{\Lambda}]v^{-p_2}(y)|y|^{\beta}dy
    $$
    may happen at $0$ or $\infty$. Thus, we write
    \begin{equation*}
    \begin{aligned}
    t(x)=&\int_{\mathbb{R}^n\setminus B_R(0)} |x-y|^{\Lambda} v^{-p_2}(y)|y|^{\beta}dy\\
    &+\int_{B_R(0) \setminus B_\delta(0)} |x-y|^{\Lambda} v^{-p_2}(y)|y|^{\beta}dy
    +\int_{B_\delta(0)} |x-y|^{\Lambda} v^{-p_2}(y)|y|^{\beta}dy\\
    :=&t_1(x)+t_2(x)+t_3(x).
    \end{aligned}
    \end{equation*}
    Here $R>2$ and $\delta \in (0,1/2)$ are constants.
    	
		When $y\in \mathbb{R}^n\setminus B_R(0)$, by \eqref{601} we have
		\begin{equation*}
			|x-y|^{\Lambda-1}v^{-p_2}(y)|y|^{\beta}
			\leq C |y|^{(1-p_2)(\Lambda+\beta)-1} \in L^1(\mathbb{R}^n\setminus B_R(0)).			
		\end{equation*}
    Here the integrability is implied by
    \begin{equation}\label{vb7}
    n<(p_2-1)(\Lambda+\beta),
    \end{equation}
    which is deduced from \eqref{j4} and $\alpha<n/(p_1-1)$.
    Therefore, according to the differentiability theorem of integrals with parameter variables
    (cf. Theorem 3.16 in \cite{BJ}), $t_1$ is differentiable at $x$ and
    $$
    \nabla t_1(x)= \int_{\mathbb{R}^n \setminus B_R(0)}
    [\nabla_x |x-y|^{\Lambda}] v^{-p_2}(y)|y|^{\beta}dy.
    $$		
    When $y\in B_\delta(0)$, in view of \eqref{602}, we have
		\begin{equation*}
			|x-y|^{\Lambda-1}v^{-p_2}(y)|y|^{\beta}
			\leq C v^{-p_2}(y)|y|^{\beta} \in L^1(B_\delta(0)).			
		\end{equation*}
		Theorem 3.16 in \cite{BJ} still shows that $t_3$ is differentiable at $x$ and
    $$
    \nabla t_3(x)= \int_{\mathbb{R}^n \setminus B_R(0)}
    [\nabla_x |x-y|^{\Lambda}] v^{-p_2}(y)|y|^{\beta}dy.
    $$		

	Clearly, $t_2$ is differentiable at $x$ and
    $$
    \nabla t_2(x)= \int_{B_R(0) \setminus B_\delta(0)}
    [\nabla_x |x-y|^{\Lambda}] v^{-p_2}(y)|y|^{\beta}dy.
    $$		
    Combining the argument about $t_1,t_2$ and $t_3$, we know that
    $t(x)$ is differentiable at $x$ and $\nabla t(x)=F(x)$.
		
		\textit{Case 2.}  When $0<\Lambda<1$, the defects of
        $F(x)$ may happen at $0$, $\infty$ or $x$.
		
		In view of $x \neq 0$, we can find $R>5$ and $\delta \in (0,1/2)$ such that $x \in B_{R/2}(0)
        \setminus B_{2\delta}(0)$. Thus, by \eqref{602},
		\begin{equation*}
			|x-y|^{\Lambda-1}v^{-p_2}(y)|y|^{\beta}
			\leq Cv^{-p_2}(y)|y|^{\beta} \in
        L^1([\mathbb{R}^n \setminus B_R(0)] \cup B_{\delta}(0)).
		\end{equation*}			
		Therefore, Theorem 3.16 in \cite{BJ} shows that both
		$$
		T_1(x):=\int_{\mathbb{R}^n\setminus B_R(0)}|x-y|^{\Lambda}v^{-p_2}(y)|y|^{\beta}dy
		$$
		and
		$$
		T_2(x):=\int_{B_{\delta}(0)}|x-y|^{\Lambda}v^{-p_2}(y)|y|^{\beta}dy
		$$
		are differentiable at $x$, and
		\begin{equation}\label{702}
			\nabla T_1(x)=\int_{\mathbb{R}^n \setminus B_R(0)}\nabla_x |x-y|^{\Lambda}v^{-p_2}(y)|y|^{\beta}dy,
		\end{equation}
		\begin{equation}\label{704}
			\nabla T_2(x)=\int_{B_\delta(0)}\nabla_x |x-y|^{\Lambda}v^{-p_2}(y)|y|^{\beta}dy.
		\end{equation}
		
		Next, we claim that
		$$
		T_3(x):=\int_{B_R(0) \setminus B_\delta(0)}|x-y|^{\Lambda}v^{-p_2}(y)|y|^{\beta}dy
		$$
		is differentiable at $x$. Here the ideas in \S 4.2 of \cite{GT} are applied.		
		Set
		$$
		\omega(x):=\int_{B_R(0) \setminus B_\delta(0)}[\nabla_x|x-y|^{\Lambda}]v^{-p_2}(y)|y|^{\beta}dy.
		$$
		Take smooth function $\zeta(x) \in C^1(\mathbb{R})$ satisfying
		$$
		\begin{cases}
			\zeta (x)=0,&\text{ $ for \left | x \right | \le 1; $ } \\
			\zeta (x) \in [0,1],&\text{ $ for \left | x \right | \in [1,2]; $ } \\
			\zeta (x)=1,&\text{ $ for \left | x \right | \ge 2, $ }
		\end{cases}
		$$
		and	$0\leq \zeta'\leq 2$. Then we have the cutting-off function
		$\zeta _{\varepsilon}(x)=\zeta ({x}/{\varepsilon}),
		$
		where $\varepsilon>0$ is sufficiently small.
		
		Define
		$$
		T_3^{\varepsilon}(x):=\int_{B_R(0)\setminus B_\delta(0)}|x-y|^{\Lambda}
        \zeta_{\varepsilon}(|x-y|)v^{-p_2}(y)|y|^{\beta}dy.
		$$		
		Clearly, after cut off the singularity, $T_3^{\varepsilon}(x)$ is differentiable at $x$, and
		$$
		\nabla T_3^{\varepsilon}(x)
		=\int_{B_R(0)\setminus B_\delta(0)}\nabla_x[|x-y|^{\Lambda}
        \zeta_{\varepsilon}(|x-y|)]v^{-p_2}(y)|y|^{\beta}dy.
		$$
		Thus, it follows
		$$
		T_3(x)-T_3^{\varepsilon}(x)
		=\int_{|x-y|\leq 2\varepsilon}
		(1-\zeta_{\varepsilon})|x-y|^{\Lambda}v^{-p_2}(y)|y|^{\beta}dy,
		$$
		and
		$$
		\omega(x)-\nabla T_3^{\varepsilon}(x)
		=\int_{|x-y|\leq 2\varepsilon}\nabla_x[(1-\zeta_{\varepsilon})|x-y|^{\Lambda}]
		v^{-p_2}(y)|y|^{\beta}dy.
		$$
		Noting
		$$
		|T_3(x)-T_3^{\varepsilon}(x)|
		\leq C\varepsilon^\Lambda \int_{|x-y|\leq 2\varepsilon}v^{-p_2}(y)|y|^{\beta}dy
		$$
		and
		$$
		\begin{aligned}
			|\omega(x)-\nabla T_3^{\varepsilon}(x)|
			\leq & \int_{|x-y|\leq 2\varepsilon}
			(C \varepsilon^{-1}|x-y|^{\Lambda}+\Lambda|x-y|^{\Lambda-1})v^{-p_2}(y)|y|^{\beta}dy\\
			\leq &C\varepsilon^{n+\Lambda-1}\sup_{B_R(0)\setminus B_\delta(0)}[{v^{-p_2}(x)|x|^{\beta}}],
		\end{aligned}
		$$
		we see that $T_3^{\varepsilon}$ and $\nabla T_3^{\varepsilon}$
        converge uniformly to $T_3$ and $\omega$ respectively when $\varepsilon\to 0$.
        Hence, the claim holds and
		\begin{equation}\label{703}
			\nabla T_3(x)=\int_{B_R(0)\setminus B_\delta(0)}[\nabla_x|x-y|^{\Lambda}]v^{-p_2}(y)|y|^{\beta}dy.
		\end{equation}
        Thus, $t=T_1+T_2+T_3$ is differentiable at $x$ when $0<\Lambda<1$.
        In addition, \eqref{702}-\eqref{703} imply \eqref{td}.
		
		Therefore, Cases 1 and 2 show that $t$ is differentiable in $\mathbb{R}^n \setminus \{0\}$ and
        \eqref{td} holds.

        Similar to the argument on $t$,
        we also derive that $s$ is continuous in $\mathbb{R}^n$ and differentiable
        in $\mathbb{R}^n \setminus \{0\}$, and \eqref{sd} holds.
    \end{proof}

    \begin{lemma}\label{le4}
    	Assume that the assumptions of Lemma \ref{lem3} are true.
    	Then
    \begin{equation}\label{s}
    		|\nabla t(x)|, |\nabla s(x)|\leq
    		\begin{cases}
    			C(1+|x|^{\Lambda-1}),
             & when \ x\in B_{h_0}(0)\setminus \{0\},\\[3mm]
                C(h_0,R_0), & when \ x\in B_{R_0}(0)\setminus B_{h_0}(0),\\[3mm]
    			C|x|^{\Lambda-1},
    & when \ x \in \mathbb{R}^n\setminus B_{R_0}(0).
    		\end{cases}
    	\end{equation}
    	Here $0 <h_0<1/3$, $R_0>2$, and $C>0$ is a constant.
    \end{lemma}
    \begin{proof}
    We only prove that $|\nabla s|$ is controlled by the right hand side of \eqref{s},
    and the proof on $|\nabla t|$ is analogous.

    	From \eqref{sd} it follows
    	\begin{equation}\label{z}
    		|\nabla s(x)| \leq \Lambda S(x)
    :=\Lambda\int_{\mathbb{R}^n}|x-y|^{\Lambda-1}u^{-p_1}(y)|y|^{\alpha}dy.
    	\end{equation}
    	
    	\textit{Case 1.} $x\in B_{h_0}(0)\setminus \{0\}$.

    \textit{Subcase 1.1.} $\Lambda \geq 1$.
    	
    	Clearly, we have
    	\begin{equation}\label{z31}
    		\begin{aligned}
    			S(x)=&\int_{B_{2|x|}(x)}|x-y|^{\Lambda-1}u^{-p_1}(y)|y|^{\alpha}dy\\
    		&+\int_{B_1(0)\setminus B_{2|x|}(x)}|x-y|^{\Lambda-1}u^{-p_1}(y)|y|^{\alpha}dy\\
    &+\int_{\mathbb{R}^n\setminus B_1(0)}|x-y|^{\Lambda-1}u^{-p_1}(y)|y|^{\alpha}dy\\
    		:=&S_1(x)+S_2(x)+S_3(x).
    		\end{aligned}
    	\end{equation}
    	
    	When $y\in B_{2|x|}(x)$, $|y| \leq 3|x|$, which leads to $|x-y| \leq 4|x|$. Therefore,
    by $u^{-p_1}(y)|y|^{\alpha} \in L^1(\mathbb{R}^n)$ (implied by \eqref{602}), we get
        \begin{equation}\label{z4}
    		S_1(x)\leq C|x|^{\Lambda-1}\int_{B_{2|x|}(x)}u^{-p_1}(y)|y|^{\alpha}dy
    		\leq  C|x|^{\Lambda-1}.
        \end{equation}

    	In view of $\alpha<n/(p_1-1)$ and $\Lambda>1$, we have
    \begin{equation*}
    n+\Lambda-1+(1-p_1)\alpha>0.
    \end{equation*}
    When $y\in B_1(0) \setminus B_{2|x|}(x)$, there holds $|y|\geq |x|$ and hence
    $|x-y| \leq |x|+|y| \leq 2|y|$. In view of \eqref{601},
    	we see $u(x)\geq C^{-1}|x|^{\alpha}$. Thus, we have
    \begin{equation}\label{vn32}
    S_2(x)
    		\leq C\int_{B_1(0)}
    		|y|^{\Lambda-1-p_1\alpha+\alpha}dy
    		\leq C.
    \end{equation}

    	Next, noticing \eqref{j4} and $0\leq \beta <{n}/(p_2-1)$,
    	we get ${1}/(p_1-1)< (\Lambda+\alpha)/(n-1)$, which leads to
    \begin{equation}\label{bn9}
    \Lambda-1-p_1(\Lambda+\alpha)+\alpha+n<0.
    \end{equation}   	
    	Therefore, by \eqref{601},
    	\begin{equation}\label{z5}
    		S_3(x) \leq C\int_{\mathbb{R}^n\setminus B_1(0)}
    |y|^{\Lambda-1-p_1(\Lambda+\alpha)+\alpha}dy \leq C.
    	\end{equation}
    	Inserting \eqref{z4}, \eqref{vn32} and \eqref{z5} into \eqref{z31}, we obtain
    	$$
    		|\nabla s(x)|\leq C,
        \quad x\in B_{h_0}(0)\setminus \{0\}.
    	$$

    \textit{Subcase 1.2.} $0<\Lambda<1$.

        Sometimes we denote $B_r(x)$ by $B(x,r)$.

    Clearly,
    	\begin{equation}\label{z3}
    		\begin{aligned}
    			S(x)=&\int_{B(x,|x|/2)}|x-y|^{\Lambda-1}u^{-p_1}(y)|y|^{\alpha}dy\\
    		&+\int_{B_1(0)\setminus B(x,|x|/2)}|x-y|^{\Lambda-1}u^{-p_1}(y)|y|^{\alpha}dy\\
    &+\int_{\mathbb{R}^n\setminus B_1(0)}|x-y|^{\Lambda-1}u^{-p_1}(y)|y|^{\alpha}dy\\
    		:=&S_4(x)+S_5(x)+S_6(x).
    		\end{aligned}
    	\end{equation}

    When $y \in B(x,|x|/2)$, $|x|/2 \leq |y| \leq 3|x|/2$. Therefore, by \eqref{601},
    $u^{-p_1}(y) \leq C|x|^{-p_1\alpha}$. Thus,
    $$
    S_4(x) \leq C|x|^{(1-p_1)\alpha}\int_{B(x,|x|/2)}|x-y|^{\Lambda-1}dy
    \leq C |x|^{(1-p_1)\alpha+n+\Lambda-1}.
    $$

    When $y \in B_1(0)\setminus B(x,|x|/2)$, $|x-y| \geq |x|/2$.
    Therefore, by \eqref{601} and $\alpha<n/(p_1-1)$, it follows
    $$
    S_5(x) \leq C|x|^{\Lambda-1} \int_{B_1(0)}|y|^{(1-p_1)\alpha}dy
    \leq C|x|^{\Lambda-1}.
    $$

    When $y \in \mathbb{R}^n \setminus B_1(0)$, by \eqref{601} and \eqref{bn9},
    $$
    S_6(x) \leq C \int_{\mathbb{R}^n \setminus B_1(0)} |y|^{\Lambda-1-p_1(\Lambda+\alpha)+\alpha}dy
    \leq C.
    $$
    Combining the estimates of $S_i(x)$ $(i=4,5,6)$, we obtain
    $$
    S(x) \leq C|x|^{\Lambda-1}, \quad x \in B_{h_0}(0)\setminus \{0\}.
    $$

    Now, Subcases 1.1 and 1.2 imply
    \begin{equation}\label{z6}
    		|\nabla s(x)|\leq C(1+
        |x|^{\Lambda-1}), \quad x\in B_{h_0}(0) \setminus \{0\}.
    	\end{equation}
    	
    	\textit{Case 2.} $x\in \mathbb{R}^n\setminus B_{R_0}(0)$.
    	
    	\textit{Subcase 2.1.} $\Lambda \geq 1$.
    	
    	In view of \eqref{600} and \eqref{601}, there holds
    	$$
    	\begin{aligned}
    		\frac{S(x)}{|x|^{\Lambda -1}}
    		\leq& \Lambda \int_{\mathbb{R}^n}(1+|y|^{\Lambda -1})u^{-p_1}(y)|y|^{\alpha}dy\\
    		\leq& C\int_{\mathbb{R}^n}(1+|y|^{\Lambda })u^{-p_1}(y)|y|^{\alpha}dy \leq C.
    	\end{aligned}
    	$$
    	Namely,
    	$$
    		S(x)\leq C|x|^{\Lambda-1}, \quad x\in \mathbb{R}^n\setminus B_{R_0}(0).
    	$$

    	\textit{Subcase 2.2.} $0<\Lambda<1$.
    	
        Clearly, we have
        \begin{equation}\label{z1}
    		\begin{aligned}
    		\frac{S(x)}{|x|^{\Lambda -1}}
    		\leq & \Lambda \int_{\mathbb{R}^n}\frac{|x|^{1-\Lambda }}{|x-y|^{1-\Lambda }}u^{-p_1}(y)|y|^{\alpha}dy\\
    		=&\Lambda \int_{\mathbb{R}^n\setminus B(x,\frac{|x|}{2})}
              \frac{|x|^{1-\Lambda }}{|x-y|^{1-\Lambda }}u^{-p_1}(y)|y|^{\alpha}dy\\
    		&+\Lambda \int_{ B(x,\frac{|x|}{2})}\frac{|x|^{1-\Lambda }}{|x-y|^{1-\Lambda }}u^{-p_1}(y)|y|^{\alpha}dy\\
    		:=&\Lambda F_1(x)+\Lambda F_2(x).
    	\end{aligned}
        \end{equation}
    	
    	When $y \in \mathbb{R}^n\setminus B(x,\frac{|x|}{2})$, there holds $|x-y| \geq {|x|}/{2}$.
        By \eqref{602} we get
    	\begin{equation}\label{613}
    		F_1(x)\leq C\int_{\mathbb{R}^n\setminus B(x,\frac{|x|}{2})}u^{-p_1}(y)|y|^{\alpha}dy \leq C.
    	\end{equation}

    	When $y \in B(x,\frac{|x|}{2})$, there holds $|x|/2 \leq |y| \leq {3|x|}/{2}$. Therefore,
    	by \eqref{601} we have
    	$$    	    		
        F_2(x) \leq C|x|^{1-\Lambda-p_1(\Lambda +\alpha)+\alpha}\int_{B(x,\frac{|x|}{2})}\frac{dy}{|x-y|^{1-\Lambda}}
        \leq C|x|^{n-p_1(\Lambda +\alpha)+\alpha}.    	
    	$$
    Noticing $\beta<n/(p_2-1)$, from
    \eqref{j4} we deduce
    $
    1/(p_1-1)<(\Lambda+\alpha)/n<(\Lambda+\alpha)/(n-\Lambda),
    $
    which implies
    $
    n-p_1(\Lambda +\alpha)+\alpha<0.
    $
    Therefore, by $|x| \geq R_0>2$, it follows
    $F_2(x) \leq C$.
    	Inserting this result and \eqref{613} into \eqref{z1}, we get
    	$$
    	S(x)\leq C|x|^{\Lambda-1}, \quad x\in \mathbb{R}^n\setminus B_{R_0}(0).
    	$$
    	
    	By Subcases 2.1 and 2.2, we obtain that
    	\begin{equation}\label{z7}
    		|\nabla s(x)|\leq C
        |x|^{\Lambda-1}, \quad x\in \mathbb{R}^n\setminus B_{R_0}(0).
    	\end{equation}

    	Obviously, $|\nabla s(x)| \leq C(h_0,R_0)$ when $x \in B_{R_0}(0) \setminus
    B_{h_0}(0)$. Combining this result with \eqref{z6} and \eqref{z7}, we derive \eqref{s}.
    	This completes the proof of Lemma \ref{le4}.
    \end{proof}

    To state the next results we need some notations.
    First, for a given real number $\lambda$, define
    $$
    T_{\lambda}=\{x \in \mathbb{R}^n|x_1=\lambda\}, \quad
    \Sigma_{\lambda}=\left\lbrace  x\in \mathbb{R}^n|x_1 \leq \lambda \right\rbrace.
    $$
    Write $x_{\lambda}=\{2\lambda-x_1, x_2, \cdots, x_n\}$ and $f_{\lambda}(x)=f(x_{\lambda})$.

    \begin{lemma}\label{le2}
        Let $(u,v)$ be a pair of positive solutions of \eqref{j3}.
         Then
        \begin{equation}\label{603}
            t(x)-t_{\lambda}(x)=
            \int_{\Sigma_{\lambda}}G(x,y)(s_{\lambda}^{-p_2}(y)
            |y_{\lambda}|^{(1-p_2)\beta}-s^{-p_2}(y)|y|^{(1-p_2)\beta})dy,
        \end{equation}
        \begin{equation}\label{604}
            s(x)-s_{\lambda}(x)=
            \int_{\Sigma_{\lambda}}G(x,y)(w_{\lambda}^{-p_1}(y)-w^{-p_1}(y))dy,
        \end{equation}
        where
        $G(x,y)=|x_{\lambda}-y|^{\Lambda }-|x-y|^{\Lambda }$.
    \end{lemma}

    \begin{proof}
        From \eqref{j3}, we have
        $$
        \begin{aligned}
            t(x)=&\int_{\Sigma_{\lambda}}|x-y|^{\Lambda }s^{-p_2}(y)|y|^{(1-p_2)\beta}dy
            +\int_{\mathbb{R}^n\setminus \Sigma_{\lambda}}|x-y|^{\Lambda }s^{-p_2}(y)|y|^{(1-p_2)\beta}dy\\
            =&\int_{\Sigma_{\lambda}}|x-y|^{\Lambda }s^{-p_2}(y)|y|^{(1-p_2)\beta}dy
            +\int_{\Sigma_{\lambda}}|x-y_{\lambda}|^{\Lambda }s_{\lambda}^{-p_2}(y)|y_{\lambda}|^{(1-p_2)\beta}dy.
        \end{aligned}
        $$
        Similarly,
        $$
            t_{\lambda}(x)
            =\int_{\Sigma_{\lambda}}|x_{\lambda}-y|^{\Lambda }s^{-p_2}(y)|y|^{(1-p_2)\beta}dy
            +\int_{\Sigma_{\lambda}}|x_{\lambda}-y_{\lambda}|^{\Lambda}
            s_{\lambda}^{-p_2}(y)|y_{\lambda}|^{(1-p_2)\beta}dy.
        $$
        Noticing $|x_{\lambda}-y|=|x-y_{\lambda}|$ and $|x_{\lambda}-y_{\lambda}|=|x-y|$, we have \eqref{603}.

        Similarly, we can also obtain \eqref{604}.
    \end{proof}

\section{Proof Theorem \ref{Rth3}}

In this section, we prove the radial symmetry of solutions of \eqref{j3} by the method of moving planes.

        {\bf Step 1.} For sufficiently negative $\lambda\leq N<0$, we claim that
        \begin{equation}\label{Step1}
            t(x)\geq t_{\lambda}(x), \quad s(x)\geq s_{\lambda}(x), \quad w(x)\geq w_{\lambda}(x),
            \quad for \; all \; x\in \Sigma_{\lambda}.
        \end{equation}

        First, we consider the relationship between $t(x)$ and $t_{\lambda}(x)$.

        We claim that
        \begin{equation}\label{606}
            \lim_{|x|\to \infty}|x|^{-\Lambda }(\nabla t(x)\cdot x-\frac{\Lambda }{2}t(x))>0.
        \end{equation}
        Here, we consider $i=1$ and fix $x_2, x_3, \cdots, x_n$.
        According to Lemma \ref{lem3}, there holds
        $$
            \left|\frac{\partial t(x)}{\partial x_1}\right|
            =\left|\Lambda \int_{\mathbb{R}^n}|x-y|^{\Lambda-1}
            \frac{(x_1-y_1)}{|x-y|}v^{-p_2}(y)|y|^{\beta}dy\right|.
        $$
        Write
        $$
        F(x,y):=\frac{|x-y|^{\Lambda-1}(x_1-y_1)}{|x_1|^{\Lambda-1}|x-y|}.
        $$
        Then for a.e. $y \in \mathbb{R}^n$, there holds
        $$
        \lim_{|x_1| \to \infty} F(x,y)=1.
        $$
        Therefore, when $|x_1|$ is sufficiently large,
        $$
        |F(x,y)v^{-p_2}(y)|y|^{\beta}|
        \leq 2v^{-p_2}(y)|y|^{\beta} \in L^1(\mathbb{R}^n).
        $$
        Applying the dominated convergence theorem and \eqref{602}, we have
        \begin{equation}
            \lim_{|x_1|\to \infty}\frac{|\frac{\partial t(x)}{\partial x_1}|}{|x_1|^{\Lambda -1}}
            = \Lambda\int_{\mathbb{R}^n}v^{-p_2}(y)|y|^{\beta}dy
            =\Lambda\lim_{|x| \to \infty}\frac{t(x)}{|x|^{\Lambda }},
        \end{equation}
        which implies \eqref{606}.

        Clearly, \eqref{606} shows that there exists sufficiently negative $N<0$ such that
        \begin{equation}\label{605}
            \nabla(|x|^{-\frac{\Lambda }{2}}t(x))\cdot x
            =|x|^{-\frac{\Lambda }{2}}(\nabla t(x)\cdot x-\frac{\Lambda }{2}t(x))>0,
            \quad \forall \; x \in \Sigma_N.
        \end{equation}
        When $x_1 \leq N$ and $2\lambda-x_1 \leq N$, from \eqref{605},
        it follows that for $x\in \Sigma_{\lambda}$,
        $$
        |x|^{-\frac{\Lambda }{2}}t(x)>|x_{\lambda}|^{-\frac{\Lambda }{2}}t_{\lambda}(x),
        $$
        which implies
        $$
        t(x)\geq t_{\lambda}(x), \quad \forall x\in \Sigma_{\lambda}.
        $$
        When $x_1 \leq N$ and $2\lambda-x_1 \geq N$, from \eqref{601}, there holds
        $$
        \begin{aligned}
            t(x)\geq & \frac{1+|x_1|^{\Lambda }}{C_1}
            \geq \frac{1+|\lambda|^{\Lambda }}{C_1}\\
            \geq &C_1(1+|N|^{\Lambda })
            \geq C_1(1+|2\lambda-x_1|^{\Lambda })
            \geq t_{\lambda}(x).
        \end{aligned}
        $$
        Thus, for all $x \in \Sigma_{\lambda}$,
        $$
        t(x)\geq t_{\lambda}(x).
        $$
        In view of $t(x)=|x|^{(1/p_1-1)\alpha}w(x)$ and $p_1>1$, we obtain
        $$
        w(x)\geq w_{\lambda}(x), \quad \forall \; x\in \Sigma_{\lambda}.
        $$
        Inserting this result into \eqref{604}, we have
        $$
        s(x) \geq s_{\lambda}(x), \quad \forall \; x\in \Sigma_{\lambda}.
        $$
        Thus, we obtain \eqref{Step1}.

        {\bf Step 2.} According to Step 1, we can start from the negative infinity of the $x_1$-direction
        and move the plane $T_\lambda$ to the right as long as \eqref{Step1} holds.
        Define
        $$
        \lambda_0=\sup\{\mu<0; s(x)\geq s_{\lambda}(x), w(x)\geq w_{\lambda}(x)
        \;for \; all \; \lambda<\mu, \;x_1\leq \lambda\}.
        $$

        For $x\in \Sigma_{\lambda_0}$, there are the following four cases:

        (i):  $w(x)=w_{\lambda_0}(x)$, $s(x)=s_{\lambda_0}(x)$;

        (ii): $w(x)=w_{\lambda_0}(x)$, $s(x)>s_{\lambda_0}(x)$;

        (iii): $w(x)>w_{\lambda_0}(x)$, $s(x)=s_{\lambda_0}(x)$;

        (iv): $w(x)>w_{\lambda_0}(x)$, $s(x)>s_{\lambda_0}(x)$.

        When (i) is true, we are done. In view of \eqref{604},
        (ii) and (iii) cannot happen.
        In the following, we prove that (iv) is also false.

        We carry out the argument by contradiction.
        Assume that (iv) is true.

        When $\lambda_0=0$, by the same argument above we can also move $T_\lambda$ from
        positive infinity of the $x_1-$direction to the left to $T_0$, and hence
        obtain the opposite conclusion to (iv). It is impossible.

        When $\lambda_0<0$, there exists an $\varepsilon_1 \in (0,1/4)$ such that
        for all $\lambda_0+\varepsilon_1<0$.
        We will show that there exists an $\varepsilon \in (0,\varepsilon_1)$ such that
        for all $\lambda_0<\lambda<\lambda_0+\varepsilon$,
        \begin{equation}\label{616}
            s(x)\geq s_{\lambda}(x), \;\; w(x)\geq w_{\lambda}(x),
            \quad \forall \; x\in \Sigma_{\lambda},
        \end{equation}
        which contradicts with the definition of $\lambda_0$.
        Therefore, (iv) is also not true, and hence we obtain
        \begin{equation}\label{Imp}
        s(x) \equiv s_{\lambda_0}(x),  \quad w(x)\equiv w_{\lambda_0}(x),
        \quad x \in \Sigma_{\lambda_0}.
        \end{equation}

        We divide the proof of \eqref{616} into three cases.
        Let $\lambda \in (\lambda_0, \lambda_0+\varepsilon_1)$.

        \textit{Case 1:} $x \in \Sigma_{\lambda} \setminus B(0,R_0)$,
        where $R_0 > 4(|\lambda_0|+1)$ is suitably large.

        Now, $|x_{\lambda}|\geq R_0/2$. By the same derivation of \eqref{605}, there holds
        $$
        |x|^{-\frac{\Lambda }{2}}t(x)>|x_{\lambda}|^{-\frac{\Lambda }{2}}t_{\lambda}(x).
        $$
        In view of $|x|>|x_\lambda|$ and $t(x)=|x|^{(1/p_1-1)\alpha}w(x)$, we get
        $$
        w(x)\geq w_{\lambda}(x), \quad |x|\geq R_0.
        $$

        \textit{Case 2:} $x \in \Sigma_{\lambda_0-1} \cap B(0,R_0)$.

        In view of $p_2>1$, by (iv) we have
        \begin{equation*}
        s_{\lambda_0}^{-p_2}(x)|x_{\lambda_0}|^{(1-p_2)\beta}-s^{-p_2}(x)|x|^{(1-p_2)\beta} > 0,
        \quad w_{\lambda_0}^{-p_1}(x)-w^{-p_1}(x)>0.
        \end{equation*}
        In addition, we have
        \begin{equation*}
        |x_{\lambda_0}-y|^{\Lambda }-|x-y|^{\Lambda }>0, \quad
        \forall  \ y \in \Sigma_{\lambda_0}.
        \end{equation*}
        Inserting these results into \eqref{603} and \eqref{604}, we obtain that for each
        $x \in \Sigma_{\lambda_0-1} \cap B(0,R_0)$, there exists $\sigma_x>0$ such that
        \begin{equation}\label{vb5}
        t(x)-t_{\lambda_0}(x) \geq \sigma_x, \quad s(x)-s_{\lambda_0}(x) \geq \sigma_x.
        \end{equation}

        Since $t$ and $s$ are continuous with respect to $\lambda$, we can find
        a suitably small constant $\varepsilon_2>0$ such that for $\lambda \in (\lambda_0,
        \lambda_0+\varepsilon_2)$,
        $$
        |t_{\lambda_0}(x)-t_{\lambda}(x)| \leq \frac{\sigma_x}{2},\quad
        |s_{\lambda_0}(x)-s_{\lambda}(x)| \leq \frac{\sigma_x}{2}.
        $$
        Combining with \eqref{vb5}, we see that
        \begin{equation}\label{617}
            t(x)-t_{\lambda}(x)\geq 0
        \end{equation}
        and
        \begin{equation}\label{622}
            s(x)-s_{\lambda}(x)\geq 0,
        \end{equation}
        for all $x \in \Sigma_{\lambda_0-1} \cap B(0,R_0)$,
        and $\lambda_0 \leq \lambda \leq \lambda_0+\varepsilon_2$.
        Therefore, from  $t(x)=|x|^{(1/p_1-1)\alpha}w(x)$ and \eqref{617}, it still follows
        \begin{equation*}
            w(x)-w_{\lambda}(x) \geq 0
        \end{equation*}
        for all $x \in \Sigma_{\lambda_0-1} \cap B(0,R_0)$, and $\lambda_0 \leq \lambda
        \leq \lambda_0+\varepsilon_2$.

        \textit{Case 3:} $x \in (\Sigma_{\lambda} \setminus \Sigma_{\lambda_0-1})\cap B(0,R_0)$.

        To deduce $t(x) \geq t_{\lambda}(x)$, by Lemma \ref{le2} we get
        \begin{equation}\label{vb1}
        \begin{aligned}
            &t(x)-t_{\lambda}(x)\\=
            &\int_{\Sigma_{\lambda}}G(x, y)(s_{\lambda}^{-p_2}(y)
            |y_{\lambda}|^{(1-p_2)\beta}-s^{-p_2}(y)|y|^{(1-p_2)\beta})dy\\
            \geq &\int_{\Sigma_{\lambda}\setminus \Sigma_{\lambda_0-1}}
            G(x,y)(s_{\lambda}^{-p_2}(y)|y_{\lambda}|^{(1-p_2)\beta}-s^{-p_2}(y)|y|^{(1-p_2)\beta})dy\\
            &+\int_{(\Sigma_{\lambda_0-2}\setminus \Sigma_{\lambda_0-3}) \cap B(0,R_0)}
            G(x,y)(s_{\lambda}^{-p_2}(y)|y_{\lambda}|^{(1-p_2)\beta}-s^{-p_2}(y)|y|^{(1-p_2)\beta})dy\\
            :=&J_1(x)+J_2(x).
        \end{aligned}
        \end{equation}
        We will verify $J_2(x) \geq 0$ and $J_1(x)+J_2(x) \geq 0$.

        {\it Estimate of} $J_2(x)$.

        Take $\widetilde{y} \in T_{\lambda_0-3}$. By \eqref{622},
        we can find $\sigma>0$ such that for any $y \in (\Sigma_{\lambda_0-2}
        \setminus\Sigma_{\lambda_0-3}) \cap B(0,R_0)$,
        \begin{equation}\label{623}
            s_{\lambda}^{-p_2}(y)|y_{\lambda}|^{(1-p_2)\beta}-s^{-p_2}(y)|y|^{(1-p_2)\beta}
            \geq s^{-p_2}(\widetilde{y})(|\widetilde{y}_{\lambda}|^{(1-p_2)\beta}
            -|\widetilde{y}|^{(1-p_2)\beta})
            \geq \sigma.
        \end{equation}
        In addition, for any
        $$
        x \in (\Sigma_{\lambda}\setminus\Sigma_{\lambda_0-1}) \cap B(0,R_0) \quad
        and
        \quad
        y \in (\Sigma_{\lambda_0-2}
        \setminus\Sigma_{\lambda_0-3}) \cap B(0,R_0),
        $$
        we have
        $$
        1 \leq |\xi-y| \leq \max\{2R_0,5\}
        $$
        when $\xi=\theta x_\lambda +(1-\theta)x$ (here $\theta \in (0,1)$).
        Therefore, using the mean value theorem, we can find $\eta>0$ such that
        \begin{equation*}
            G(x,y) =\Lambda |\xi-y|^{\Lambda-1} (2\lambda-2x_1)
            \geq \eta(\lambda-x_1).
        \end{equation*}
        By this result and \eqref{623}, there holds
        \begin{equation}\label{621}
            J_2(x)\geq c_*\sigma\eta(\lambda-x_1),
        \end{equation}
        where $c_*, \sigma, \eta$ are positive absolute constants.

        Next, we estimate $J_1(x)$.

        In view of (iv), we have
        $$
        s_{\lambda_0}^{-p_2}(y)|y_{\lambda_0}|^{(1-p_2)\beta} \geq s^{-p_2}(y)|y|^{(1-p_2)\beta}
        $$
        when $y \in \Sigma_{\lambda}\setminus \Sigma_{\lambda_0-1}$.
        Therefore,
        \begin{equation}\label{bn2}
        \begin{aligned}
            &J_{1}(x)\\
            \geq & \int_{(\Sigma_{\lambda}\setminus \Sigma_{\lambda_0-1})\setminus B_M(0)}
            G(x,y)(s_{\lambda}^{-p_2}(y)|y_{\lambda}|^{(1-p_2)\beta}-s^{-p_2}(y)|y|^{(1-p_2)\beta})dy\\
            &+ \int_{(\Sigma_{\lambda}\setminus \Sigma_{\lambda_0-1})\cap B_M(0)}
            G(x,y)(s_{\lambda}^{-p_2}(y)|y_{\lambda}|^{(1-p_2)\beta}
            -s_{\lambda_0}^{-p_2}(y)|y_{\lambda_0}|^{(1-p_2)\beta})dy\\
            :=& J_{11}(x)+J_{12}(x),
        \end{aligned}
        \end{equation}
        where $M>R_0$ is a large constant which will be determined later.

        {\it Estimate of} $J_{11}(x)$.

        In view of
        $$
        x \in (\Sigma_{\lambda}\setminus\Sigma_{\lambda_0-1})\cap B(0,R_0) \quad
        and
        \quad
        y \in (\Sigma_{\lambda}\setminus\Sigma_{\lambda_0-1})\setminus B(0,M),
        $$
        we have $|x-y|\sim |x_{\lambda}-y|\sim|y|$ when $M$ is sufficiently large. Consequently, it follows
        $$
        \begin{aligned}
            G(x,y)
            =&|x_{\lambda}-y|^{\Lambda }-|x-y|^{\Lambda }\\
            \leq & C\max\{|x_{\lambda}-y|^{\Lambda -1}, |x-y|^{\Lambda -1}\}||x_{\lambda}-y|-|x-y||\\
            \leq & C|y|^{\Lambda -1}(\lambda-x_1),
        \end{aligned}
        $$
        where $C>0$ is independent of $M$.
        By this result and \eqref{601}, we obtain that
        $$\begin{aligned}
                |J_{11}(x)|\leq
                &\left|\int_{(\Sigma_{\lambda}\setminus \Sigma_{\lambda_0-1})\setminus B(0,M)}
                G(x,y)(s_{\lambda}^{-p_2}(y)|y_{\lambda}|^{(1-p_2)\beta}
                -s^{-p_2}(y)|y|^{(1-p_2)\beta})dy\right|\\
                \leq&C(\lambda-x_1)\int_{\mathbb{R}^n\setminus B(0,M)}
                |y|^{\Lambda -1}|y|^{-p_2\Lambda +(1-p_2)\beta}dy.
            \end{aligned}
        $$
        By \eqref{vb7}, there holds
        \begin{equation}\label{bn5}
        \Lambda -1-p_2\Lambda +(1-p_2)\beta+n<0.
        \end{equation}
        Therefore, it follows that
        $$
        |J_{11}(x)| \leq C_1 M^{\Lambda -1-p_2\Lambda +(1-p_2)\beta+n}(\lambda-x_1),
        $$
        where $C_1>0$ is an absolute constant. Take $M>0$ suitably large such that
        $$
        C_1 M^{\Lambda -1-p_2\Lambda +(1-p_2)\beta+n} \leq \frac{c_*\sigma\eta}{4}.
        $$
        Thus, there holds
        \begin{equation}\label{619}
        |J_{11}(x)| \leq \frac{c_*\sigma\eta}{4}(\lambda-x_1),
        \end{equation}

        {\it Estimate of} $J_{12}(x)$.
        Now,
        $$
        x\in (\Sigma_{\lambda}\setminus\Sigma_{\lambda_0-1})\cap B(0,R_0)\quad and \quad
        y\in (\Sigma_{\lambda}\setminus\Sigma_{\lambda_0-1})\cap B(0,M).
        $$
        Clearly, the mean value theorem implies that
        there exists $\xi$ satisfying $|x-y|<|\xi-y|<|x_\lambda-y|$ such that
        \begin{equation}\label{bn3}
        G(x,y) \leq 2\Lambda
        |\xi-y|^{\Lambda-1}(\lambda-x_1) \leq
        \begin{cases}
        	2\Lambda|x-y|^{\Lambda-1}(\lambda-x_1), &as \  0<\Lambda<1,\\
        	2\Lambda|x_{\lambda}-y|^{\Lambda-1}(\lambda-x_1), &as \ \Lambda \geq 1.
        \end{cases}
        \end{equation}

        We estimate $J_{12}$ in two cases.

        (I) $\Lambda \geq 1$.

        Clearly, \eqref{bn3} shows that
        $G(x,y)$ in the improper integrals has no singularity.
        In addition, $|y|^{(1-p_2)\beta}s^{-p_2}(y)=|y|^\beta v^{-p_2}(y) \in L^1(R^n)$ (implied by \eqref{602}).
        Therefore, by the absolute continuity of an integral,
        we can find a suitably small $h_0 \in (0,1/2)$ such that
        \begin{equation}\label{bn4}
        \begin{aligned}
        &\int_{B(0_\lambda,h_0)}[(|y_\lambda|^{(1-p_2)\beta}s_\lambda^{-p_2}(y)
        +|y_{\lambda_0}|^{(1-p_2)\beta}s_{\lambda_0}^{-p_2}(y)]dy\\
        &<[8\Lambda(2R_0+2|\lambda_0|+1)^{\Lambda-1}]^{-1}c_*\sigma\eta.
        \end{aligned}
        \end{equation}
        Now we divide $J_{12}(x)$ into two terms
        $$
        J_{12}(x)=\sum_{i=1}^2\int_{\Omega_i}
        G(x,y)(s_{\lambda}^{-p_2}(y)|y_{\lambda}|^{(1-p_2)\beta}
        -s_{\lambda_0}^{-p_2}(y)|y_{\lambda_0}|^{(1-p_2)\beta})dy
        :=H_1+H_2,
        $$
        where
         $$
         \Omega_1=[(\Sigma_{\lambda}\setminus \Sigma_{\lambda_0-1}) \cap B(0_\lambda,h_0),
         $$
         $$
         \Omega_2=(\Sigma_{\lambda}\setminus \Sigma_{\lambda_0-1}) \cap [B(0,M) \setminus B(0_\lambda,h_0)].
         $$

         When $y \in \Omega_1$, by \eqref{bn3} we see
         $$
         G(x,y) \leq 2\Lambda |x_\lambda-y|^{\Lambda-1}(\lambda-x_1)
         \leq 2\Lambda (2R_0+2|\lambda_0|+h_0)^{\Lambda-1}(\lambda-x_1).
         $$
         Therefore, by \eqref{bn4} we get
         $$
         |H_1| \leq c_*\sigma\eta(\lambda-x_1)/4.
         $$

         When $y \in \Omega_2$, there holds $|x_\lambda-y|<2M$.
         In addition, we can find $\epsilon_0 \in (0,1/3)$ such that
        $$
        |y_{\lambda}|/2 \leq |y_{\lambda_0}| \leq 2|y_{\lambda}|,
        \quad \forall \lambda \in (\lambda_0,\lambda_0+\epsilon_0).
        $$
        Noting
        \begin{equation}\label{bn8}
        \begin{aligned}
        &|s_{\lambda}^{-p_2}(y)|y_{\lambda}|^{(1-p_2)\beta}
        -s_{\lambda_0}^{-p_2}(y)|y_{\lambda_0}|^{(1-p_2)\beta}| \\
        \leq & s_{\lambda}^{-p_2}(y)||y_{\lambda}|^{(1-p_2)\beta}
        -|y_{\lambda_0}|^{(1-p_2)\beta}|
        + |y_{\lambda_0}|^{(1-p_2)\beta}|s_{\lambda}^{-p_2}(y)-s_{\lambda_0}^{-p_2}(y)|\\
        \leq & C\{s_{\lambda}^{-p_2}(y)[|y_{\lambda}|^{(1-p_2)\beta-1}
        +|y_{\lambda_0}|^{(1-p_2)\beta-1}]\\
        &+|y_{\lambda_0}|^{(1-p_2)\beta}[s_{\lambda}^{-p_2-1}(y)
        +s_{\lambda_0}^{-p_2-1}(y)]|\nabla s|\} (\lambda-\lambda_0),
        \end{aligned}
        \end{equation}
        we can deduce from \eqref{bn3}, \eqref{601} and \eqref{s}
        that for all $\lambda \in (\lambda_0,\lambda_0+\epsilon_0)$,
        $$\begin{aligned}
        |H_2| \leq & \int_{\Omega_2}2\Lambda(\lambda-x_1)|x_\lambda-y|^{\Lambda-1}
        |s_{\lambda}^{-p_2}(y)|y_{\lambda}|^{(1-p_2)\beta}
        -s_{\lambda_0}^{-p_2}(y)|y_{\lambda_0}|^{(1-p_2)\beta}|dy\\
        \leq & C(\lambda-x_1)M^{\Lambda-1}(\lambda-\lambda_0)\\
        &\cdot [\int_{B_1(0_\lambda) \setminus B_{h_0}(0_\lambda)}|y_\lambda|^{(1-p_2)\beta}
        (|y_\lambda|^{-1}+1)dy\\
        &+\int_{B_{2M}(0_\lambda) \setminus B_1(0_\lambda)}(|y_\lambda|^{(1-p_2)\beta-1-p_2\Lambda}
        +|y_\lambda|^{(1-p_2)\beta-(p_2+1)\Lambda+\Lambda-1})dy].
        \end{aligned}
        $$
        By \eqref{vb7}, it follows
        $$
        |H_2| \leq  C_2(\lambda-x_1)(\lambda-\lambda_0)M^{\Lambda-1}(1+h_0^{n+\theta_1}+M^{n+\theta_2}),
        $$
        where $C_2>0$ is an absolute constant, and
        $$
        \theta_1:=(1-p_2)\beta-1,\quad
        \theta_2:=(1-p_2)\beta-1-p_2\Lambda.
        $$

         Thus, the estimates of $H_1$ and $H_2$ show that
         \begin{equation}\label{626}
         	|J_{12}(x)|\leq [c_*\sigma\eta/4+C_2(\lambda-\lambda_0)M^{\Lambda-1}
         (1+h_0^{n+\theta_1}+M^{n+\theta_2})](\lambda-x_1).
         \end{equation}

        (II) $0<\Lambda<1$.

        (II.1) $x \neq 0_{\lambda_0}$.

        Write $r_x:=|x-0_{\lambda_0}|/2>0$. We can find small $\epsilon^* \in (0,1/4)$
        such that $0_\lambda \not \in B(x,r_x)$ for $\lambda \in (\lambda_0,\lambda_0+\epsilon^*)$.
        By the absolute continuity of an integral,
        we can also find a suitably small $h_0 \in (0,1/2)$ such that
        \begin{equation}\label{bn44}
        \int_{B(0_\lambda,h_0)}[(|y_\lambda|^{(1-p_2)\beta}s_\lambda^{-p_2}(y)
        +|y_{\lambda_0}|^{(1-p_2)\beta}s_{\lambda_0}^{-p_2}(y)]dy
        <[8\Lambda r_x^{\Lambda-1}]^{-1}c_*\sigma\eta.
        \end{equation}
        In addition, we can find $h \in (0,\min\{h_0,r_x\})$ such that
        $B(0_\lambda,h) \subset B(0,M)$ and $B(0_\lambda,h) \cap B(x,r_x)=\emptyset$.
        Now we divide $J_{12}(x)$ into three terms
        $$\begin{aligned}
        J_{12}(x)&=\sum_{i=3}^5 \int_{\Omega_i} G(x,y)(s_{\lambda}^{-p_2}(y)|y_{\lambda}|^{(1-p_2)\beta}
                -s_{\lambda_0}^{-p_2}(y)|y_{\lambda_0}|^{(1-p_2)\beta})dy\\
                &:=H_3+H_4+H_5,
                \end{aligned}
        $$
        where
        $$
        \Omega_3=[(\Sigma_{\lambda}\setminus \Sigma_{\lambda_0-1})] \cap B(0_\lambda,h),
        $$
        $$
        \Omega_4=[(\Sigma_{\lambda}\setminus \Sigma_{\lambda_0-1}) \cap B(0,M)]
        \setminus [B(x,r_x) \cup B(0_\lambda,h)],
        $$
        $$
        \Omega_5=(\Sigma_{\lambda}\setminus \Sigma_{\lambda_0-1}) \cap B(x,r_x).
        $$
        By \eqref{bn3} and \eqref{bn44} we get
        $$
        |H_3| \leq c_*\sigma\eta(\lambda-x_1)/4.
        $$
        Similar to the estimate of $H_2$, we also get
        $$
        |H_4| \leq  C(\lambda-x_1)(\lambda-\lambda_0)r_x^{\Lambda-1}(1+h^{n+\theta_1}+M^{n+\theta_2}),
        \quad \forall \lambda \in (\lambda_0,\lambda_0+\epsilon_0).
        $$

        When $y \in \Omega_5$, we see $y \in B_M(0) \setminus B_h(0_\lambda)$, which
        leads to $h \leq |y_\lambda| \leq 2M$. Therefore, by \eqref{bn8} we see that
        for all $\lambda \in (\lambda_0,\lambda_0+\epsilon_0)$,
        $$
        |H_5| \leq C (\lambda-x_1) (\lambda-\lambda_0) \left(\max_{h \leq |x| \leq 2M}L(x)\right)
        \int_{B(x,r_x)}|x-y|^{\Lambda-1}dy,
        $$
        where
        $$
        L(x):=s^{-p_2}(x) |x|^{(1-p_2)\beta-1}+|x|^{(1-p_2)\beta} s^{-p_2-1}(x)
         |\nabla s(x)|.
        $$
        By \eqref{601} and \eqref{s} we get
        $$
        \max_{h \leq |x| \leq 2M}L(x) \leq C(1+h^{\theta_1}+M^{\theta_2}).
        $$
        Therefore,
        $$
        |H_5| \leq Cr_x^{n+\Lambda-1}(\lambda-x_1)(1+h^{\theta_1}+M^{\theta_2}) (\lambda-\lambda_0).
        $$

        By the estimates of $H_i$ $(i=3,4,5)$, we can see that for
        all $\lambda \in (\lambda_0, \lambda_0+\min\{\epsilon_0,\epsilon^*\})$,
        there holds
        \begin{equation}\label{625}
        \begin{aligned}
        |J_{12}(x)| \leq (\lambda-x_1) (\frac{c_*\sigma\eta}{4}+
        &C_3r_x^{\Lambda-1}(\lambda-\lambda_0)\\
        &\cdot [(1+h^{\theta_1}+M^{\theta_2})+r_x^n(1+h^{n+\theta_1}+M^{n+\theta_2})] ).
        \end{aligned}
        \end{equation}
        Here $C_3>0$ is an absolute constant.

          (II.2) $x=0_{\lambda_0}$.

          Now $\xi \neq 0_{\lambda_0}$ and we replace $x$ by $\xi$ in the argument of (II.1).
          By \eqref{bn3} we can still obtain the estimates of $H_i$ ($i=1,2,3,4,5$).
        Thus, similar to \eqref{625}, we also have
        \begin{equation}\label{620}
        \begin{aligned}
        |J_{12}(x)| \leq (\lambda-x_1) (\frac{c_*\sigma\eta}{4}+
        &C_4r_\xi^{\Lambda-1}(\lambda-\lambda_0)\\
        &\cdot [(1+h^{\theta_1}+M^{\theta_2})+r_\xi^n(1+h^{n+\theta_1}+M^{n+\theta_2})] ).
        \end{aligned}
        \end{equation}
        Here $r_\xi=|\xi-0_{\lambda_0}|/2$, and $C_4>0$ is an absolute constant.

        {\it Proof of \eqref{616}.}

        By \eqref{626}, \eqref{625} and \eqref{620}, we can find
        $\varepsilon_3 \in (0, \min\{\epsilon_0,\epsilon^*,\epsilon_*\})$
        such that for all $\lambda \in (\lambda_0,\lambda_0+\varepsilon_3)$, there holds
        \begin{equation}\label{cxv}
        |J_{12}(x)| \leq \frac{c_*\sigma\eta}{2}(\lambda-x_1),
        \end{equation}
        where
        $$\begin{aligned}
        \epsilon_*=\frac{c_*\sigma\eta}{4}(\max\{&C_2M^{\Lambda-1}(1+h_0^{n+\theta_1}+M^{n+\theta_2}),\\
        &C_3r_x^{\Lambda-1}[(1+h^{\theta_1}+M^{\theta_2})+r_x^n(1+h^{n+\theta_1}+M^{n+\theta_2})],\\
        &C_4r_\xi^{\Lambda-1}[(1+h^{\theta_1}+M^{\theta_2})+r_\xi^n(1+h^{n+\theta_1}+M^{n+\theta_2})]\})^{-1}.
        \end{aligned}
        $$

        Combining \eqref{621}, \eqref{619} and \eqref{cxv} with \eqref{vb1} and \eqref{bn2},
        we obtain that for all $\lambda \in (\lambda_0,\lambda_0+\varepsilon_3)$,
        $$
        t(x)-t_{\lambda}(x) \geq (\lambda-x_1)[c_*\sigma\eta(1-1/4-1/2)] \geq 0, \quad
        x \in (\Sigma_{\lambda} \setminus \Sigma_{\lambda_0-1})\cap B(0,R_0).
        $$
        In view of the definition of $w(x)$, we see that
        $$
        w(x) \geq w_{\lambda}(x), \quad
        x \in (\Sigma_{\lambda} \setminus \Sigma_{\lambda_0-1})\cap B(0,R_0).
        $$

        Combining the results of Cases 1, 2, and 3,
        we derive that for all $\lambda \in(\lambda_0,\lambda_0+\min\{\varepsilon_1,\varepsilon_2,\varepsilon_3\})$,
        \begin{equation}\label{nm}
        w(x)\geq w_{\lambda}(x), \quad x \in \Sigma_{\lambda}.
        \end{equation}
        Inserting this result into \eqref{604}, we have
        $$
        s(x)\geq s_{\lambda}(x), \quad x \in \Sigma_{\lambda}.
        $$
        This result and \eqref{nm} imply that \eqref{616} holds.

        {\bf Step 3.} We claim $\lambda_0=0$.

        Otherwise, $\lambda_0<0$.
        Now, \eqref{Imp} holds.
        In view of $|x_{\lambda_0}-y|>|x-y|$, $|x|>|x_{\lambda_0}|$ and $|y|>|y_{\lambda_0}|$
        for $x, y\in\Sigma_{\lambda_0}$, \eqref{603} shows that
        $$
        \begin{aligned}
            0>& \ w(x)[|x|^{(1/p_1-1)\alpha}-|x_{\lambda_0}|^{(1/p_1-1)\alpha}]\\
            =&\ t(x)-t_{\lambda_0}(x)\\
            =&\int_{\Sigma_{\lambda_0}}(|x_{\lambda_0}-y|^{\Lambda }-|x-y|^{\Lambda })(s_{\lambda_0}^{-p_2}(y)|y_{\lambda_0}|^{(1-p_2)\beta}-s^{-p_2}(y)|y|^{(1-p_2)\beta})dy\\
            =&\int_{\Sigma_{\lambda_0}}
            (|x_{\lambda_0}-y|^{\Lambda }-|x-y|^{\Lambda })(|y_{\lambda_0}|^{(1-p_2)\beta}-|y|^{(1-p_2)\beta})s^{-p_2}(y)dy\\
            >&0.
        \end{aligned}
        $$
        This is impossible. It means that $\lambda_0=0$.

        {\bf Step 4.} Complete the proof of Theorem \ref{Rth3}.

        Steps 1-3 show that both $w$ and $s$ are symmetric about the plane $T_0$.
        Since we can orient the $x_1$ axis in any direction, $w$ and $s$ are radially symmetric about the origin.
        The monotonicity of $w$ and $s$ easily follows from the argument above.
        Thus, both $u$ and $v$ are radially symmetric and monotonically increasing about the origin.
  This completes the proof of Theorem \ref{Rth3}.

\paragraph{Acknowledgements.} This research was supported
by the Natural Science Foundation of Jiangsu (No. BK20241878).

\paragraph{Conflict of interest.}
This work does not have any conflicts of interest.

\paragraph{Data availability.}
There is no data associated with this work.

{\sc Tiantian Zhou}

Institute of Mathematics

School of Mathematical Sciences

Nanjing Normal University, Nanjing, 210023, China

Email:ztt0515@foxmail.com

\vskip 5mm

{\sc Yutian Lei}

Ministry of Education Key Laboratory of NSLSCS

School of Mathematical Sciences

Nanjing Normal University, Nanjing, 210023, China

Email: leiyutian@njnu.edu.cn

\end{document}